\documentclass[11pt]{amsart}
\usepackage{fullpage}
\usepackage[active]{srcltx}
\usepackage{amssymb,amsthm,amsmath,amscd}
\usepackage{amsmath,amssymb,amsthm,mathtools}
\usepackage{array,booktabs,tabularx,longtable}
\usepackage{microtype}
\usepackage{enumitem}
\usepackage[hidelinks]{hyperref}
\usepackage[nameinlink,noabbrev]{cleveref}
\usepackage{xcolor}
\usepackage{comment}

\usepackage{hyperref}

\usepackage{graphicx} 
\usepackage{amsfonts}

\usepackage{tikz-cd}
\usepackage[all,cmtip]{xy}

\newtheorem{theorem}{Theorem}
\newtheorem{lemma}[theorem]{Lemma}
\newtheorem{proposition}[theorem]{Proposition}
\newtheorem{conjecture}[theorem]{Conjecture}
\newtheorem{question}[theorem]{Question}
\newtheorem{corollary}[theorem]{Corollary}
\newtheorem{claim}[theorem]{Claim}

\theoremstyle{definition}
\newtheorem{remark}[theorem]{Remark}

\DeclareMathOperator{\m}{\mathfrak{m}}

\DeclareMathOperator{\q}{\mathfrak{q}}
\DeclareMathOperator{\sO}{\mathcal{O}}
\DeclareMathOperator{\fra}{\mathfrak{a}}
\DeclareMathOperator{\depth}{\mathrm{depth}}
\DeclareMathOperator{\Supp}{\mathrm{Supp}}

\DeclareMathOperator{\Ann}{\text{Ann}}
\DeclareMathOperator{\pd}{\text{pd}}
\DeclareMathOperator{\Ext}{\mathrm{Ext}}
\DeclareMathOperator{\Tor}{\mathrm{Tor}}
\DeclareMathOperator{\Kos}{\mathrm{Kos}}
\DeclareMathOperator{\Hom}{\mathrm{Hom}}
\DeclareMathOperator{\Image}{\text{Image}}
\DeclareMathOperator{\Spec}{\mathrm{Spec}}
\DeclareMathOperator{\Proj}{\mathrm{Proj}}
\DeclareMathOperator{\codim}{\mathrm{codim}}
\DeclareMathOperator{\grade}{\mathrm{grade}}
\DeclareMathOperator{\coker}{\mathrm{coker}}

\newcommand{\rank}{\operatorname{rank}}

\newcommand{\height}{\operatorname{ht}}

\renewcommand{\m}{\mathfrak{m}}

\title{Counterexamples to Peskine--Szpiro's conjectures on modules of finite projective dimension}
\author{Linquan Ma}
\address{Department of Mathematics, Purdue University, West Lafayette, IN 47907, USA}
\email{ma326@purdue.edu}

\begin{document}

\begin{abstract}
We give counterexamples to the remaining homological conjectures of Peskine--Szpiro from \cite{PeskineSzpiroIHES} concerning modules of finite projective dimension over Noetherian local rings. Specifically, we show that the dimension inequality, the strong intersection conjecture, and the grade conjecture are all false in general.
\end{abstract}

\maketitle

\tableofcontents

\newpage
\section{Introduction}

In their seminal work \cite{PeskineSzpiroIHES}, Peskine--Szpiro raised six interrelated conjectures concerning the dimension, the depth, and some other homological properties of modules of finite projective dimension over possibly singular Noetherian local rings. These conjectures are partially motivated by Serre's homological approach to local intersection theory \cite{SerreLocalAlgebra}, and they were among the earliest problems now grouped under the name ``homological conjectures''. The guiding idea was that a closed subset defined as the support of a module of finite projective dimension should retain, even in a singular ambient scheme, many of the intersection-theoretic features familiar from the nonsingular case. Accordingly, the six conjectures were presented as a package governing the topology of such supports. Hochster subsequently enlarged and organized this circle of questions in
\cite{HochsterTopicsCBMS} and studied them in depth, emphasizing in particular the role of big Cohen--Macaulay
modules as a common mechanism for attacking them.  Peskine--Szpiro's conjectures thus form one of the starting points of Hochster's broader program, and they have several applications to conjectures of Auslander and Bass, as well as many questions in local cohomology.

The homological conjectures have generated a tremendous amount of activity over the past fifty years. At the heart of Hochster's web are the direct summand conjecture and the big Cohen--Macaulay modules and algebras -- these were known in equal characteristic, with the mixed characteristic case being the main obstruction. About a decade ago, Andr\'{e} \cite{AndreDSC} proved the direct summand conjecture and showed the existence of big Cohen--Macaulay algebras. Together with the known equivalences and implications, these breakthroughs
settled the central part of Hochster's list of the homological conjectures.
On the other hand, they did not shed light on the following three conjectures of Peskine--Szpiro, which form a separate branch of the story and had remained
open even in equal characteristic.

\begin{conjecture}[Peskine--Szpiro]
\label{conj: Peskine--Szpiro}
Let $R$ be a Noetherian local ring and let $M$ be a finitely generated $R$-module of finite projective dimension. Suppose $N$ is a finitely generated $R$-module so that $M\otimes_R N$ has finite length. Then 
\begin{enumerate}
    \item[$(i)$] $\dim(M) + \dim(N) \leq \dim(R)$ -- \emph{(dimension inequality)};
    \item[$(ii)$]  $\dim(N) \leq \grade(M)$ -- \emph{(strong intersection
        conjecture)};
    \item[$(iii)$] $\grade(M) = \dim(R) - \dim(M)$ -- \emph{(grade conjecture)}.
\end{enumerate}
\end{conjecture}

Indeed, $(i), (ii)$, and $(iii)$ are \cite[II (a), (e), (f)]{PeskineSzpiroIHES} respectively. The other three conjectures include Auslander's zerodivisor conjecture, the rigidity of Tor, and the intersection conjecture \cite[II (b), (c), (d)]{PeskineSzpiroIHES}. These had been settled before Andr\'{e}'s work: the rigidity of Tor is false \cite{HeitmannCounterexampleRigidityofTor}, whereas the
zerodivisor and intersection conjectures are equivalent and are both true; see \cite{PeskineSzpiroIHES,HochsterTopicsCBMS,RobertsIntersectionTheorem}. It is not difficult to show (see \cite[II Th\'{e}or\`{e}me (0.10)]{PeskineSzpiroIHES}) that  
$$\text{strong intersection conjecture} \,\ \Longleftrightarrow \,\ \text{dimension inequality }  + \text{ grade conjecture}.$$

In this paper, we construct counterexamples to Conjecture~\ref{conj: Peskine--Szpiro}:

\begin{theorem}
\label{thm: main}
There exists a Noetherian complete local ring of dimension three for which the dimension inequality does not hold; and there exists a Noetherian complete local ring of dimension four for which the grade conjecture does not hold. Thus, all three statements in Conjecture~\ref{conj: Peskine--Szpiro} are false in general.
\end{theorem}

The two examples arise from a common gluing construction.  Starting with a
module of finite projective dimension from the
Dutta--Hochster--McLaughlin's example in \cite{DuttaHochsterMcLaughlin}, we obtain a perfect complex
over a two-dimensional regular local ring whose only two nonzero cohomology
modules are isomorphic finite dimensional vector spaces.  A lifting lemma
realizes this complex as a derived restriction from a power series extension,
after which Milnor patching over a fiber product ring produces the desired
module of finite projective dimension. Using one new variable gives a
counterexample to the dimension inequality, and using two variables gives a counterexample to the grade conjecture; see the beginning of Section~2 for the a more detailed explanation of the overall strategy.

On the other hand, a number of important positive results of Conjecture~\ref{conj: Peskine--Szpiro} are
known and we briefly record them here.

\begin{enumerate}
    \item All three statements hold when $M$ is perfect, by the new intersection theorem \cite{RobertsIntersectionTheorem}.
    \item All three statements hold when $R$ is $\mathbb{N}$-graded over an Artinian local ring and $M$ and $N$ are graded  \cite{PeskineSzpiroSyzygiesMultiplicities}.
    \item All three statements hold when $M$ is liftable to a regular local ring, see Corollary~\ref{cor: liftable}.
    \item The dimension inequality holds when $R$ is a hypersurface in an equal characteristic or unramified mixed characteristic regular local ring \cite{HochsterDimensionIntersectionHypersurface}.
    \item The grade conjecture holds when $R$ is equidimensional and catenary, see \cite[Proposition (A.2) and Corollary (A.3)]{AvramovFoxbyCMdefect} (see also \cite[Proposition 6.3.3]{RobertsBook}).
    \item The dimension inequality holds when $\dim(R)\leq 2$ and the grade conjecture holds when $\dim(R)\leq 3$, see Proposition~\ref{prop: small dimensions}.
\end{enumerate}

The last item shows, in particular, that Theorem~\ref{thm: main} is optimal in terms of the dimensions of the rings in the counterexamples. Even though Conjecture~\ref{conj: Peskine--Szpiro} is false in general, the following natural question remains open. 

\begin{question}
Let $R$ be a Noetherian local ring and let $M, N$ be finitely generated $R$-modules of finite projective dimension so that $M\otimes_R N$ has finite length. Does $\dim(M) + \dim(N) \leq \dim(R)$ hold?
\end{question}

The paper is organized as follows.  Section~2 explains the construction of the
counterexamples. Section~\ref{sec: DHM} is a detailed analysis on the properties of Dutta--Hochster--McLaughlin's example that will be needed henceforth; Section~\ref{sec: lifting lemma} proves an elementary lifting lemma that is needed for the gluing; Sections~\ref{sec: counterexample dimension inequality} and \ref{sec: counterexample grade conjecture} then use Milnor patching to construct the counterexamples to the dimension inequality and the grade conjecture, respectively. Appendix~\ref{app: negative intersection} revisits the negative intersection multiplicities underlying the construction, following \cite{RobertsSrinivasModulesoffinitelengthandfiniteprojectivedimension,KuranoNumericalequivalenceonChowgroupsoflocalrings}.  Appendix~\ref{app: positive results} collects some positive results, including the liftable case and the low-dimensional bounds
that show the examples are dimensionally optimal.

\subsection*{Acknowledgment} 
I would like to thank Bhargav Bhatt for many insightful discussions on the counterexamples. I would also like to thank Srikanth Iyengar, Nawaj KC and Kaito Kimura for their comments on this article and related discussions on Peskine--Szpiro's conjectures. The author is supported in part by NSF grant DMS-2302430.

\subsection*{AI disclosure}
The counterexample to the dimension inequality in Conjecture~\ref{conj: Peskine--Szpiro} was found by OpenAI's ChatGPT 5.6 Sol Pro. After several unsuccessful initial prompts, the author suggested that the AI make use of the Dutta--Hochster--McLaughlin's example in \cite{DuttaHochsterMcLaughlin} together with some gluing constructions that the AI had already been pursuing. The actual gluing strategy in Section~\ref{sec: counterexample dimension inequality}, the detailed analysis of the Dutta--Hochster--McLaughlin's example in Section~\ref{sec: DHM}, as well as the lifting lemma in Section~\ref{sec: lifting lemma} are all discovered by the AI. The author then observed that a minimal modification of the strategy also immediately yields a counterexample to the grade conjecture. The author wrote the paper, with assistance from ChatGPT in producing the data and the two tables in Section~\ref{sec: DHM} (and polishing the introduction). The materials in the two appendices are not original, and are not produced by the AI.

\newpage
\section{The counterexamples}

We first explain the overall strategy of the construction, which is actually remarkably simple. We start with a Cohen--Macaulay local ring $A$ for which there is a finitely generated $A$-module $M_A$ of finite projective dimension with $\dim(M_A)=1$, and a height one prime $P$ so that $A/P$ is regular and $\chi(M_A, A/P)=0$. It is well-known that such example exist: a slight modification of the famous counterexample of the generalized vanishing in \cite{DuttaHochsterMcLaughlin} will work (see Section~\ref{sec: DHM}). Let $C:=A/P$, $B:=C[[u]]$, and $G:= M_A\otimes_A^\mathbb{L}C$. Next, we want to construct a finite length $B$-module $M_B$ so that $M_B \otimes_B^\mathbb{L}C \cong G$ in $D(C)$. This step requires a computation, see Section~\ref{sec: lifting lemma}. Finally, we set $R:= A \times_C B$. Since by our construction, $M_A \otimes_A^\mathbb{L}C \cong M_B \otimes_B^\mathbb{L}C$, Milnor patching for perfect complexes implies that there is a perfect complex $F$ over $R$ so that $F\otimes_R^\mathbb{L}A \cong M_A$ in $D(A)$ and $F\otimes_R^\mathbb{L}B \cong M_B$ in $D(B)$. A small computation shows that $R$ is Cohen--Macaulay and $F\cong H^0(F)$ in $D(R)$. Thus $M:=H^0(F)$ is a finitely generated module of finite projective dimension over $R$ such that $\dim(M)=1$ (since $\dim(M_A)=1$ and $\dim(M_B)=0$) and $M\otimes_RB\cong M_B$ has finite length. In particular, 
$$\dim(M) + \dim(B) = 1+\dim(R) > \dim(R).$$
This gives a counterexample to the dimension inequality (and thus the strong intersection conjecture), see Section~\ref{sec: counterexample dimension inequality}. The construction of the counterexample to the grade conjecture is very similar, one simply increases the dimension of $B$: instead of using $C[[u]]$, one uses $C[[u,v]]$, see Section~\ref{sec: counterexample grade conjecture} for more details.

\subsection{Dutta--Hochster--McLaughlin's example}
\label{sec: DHM}
The goal of this section is to construct a module of finite length and finite projective dimension over a Cohen--Macaulay local ring with negative intersection multiplicity with a codimension-one quotient, and that the relevant $\Tor$'s are annihilated by the maximal ideal. We will show that the example studied in \cite{DuttaHochsterMcLaughlin} satisfies all these conditions via explicit computations.

We start by introducing some notations. Let $k$ be a field and let 
\begin{equation*}
A =\frac{k[[x,y,z,w]]}{(xw-yz)} \,\ \text{ and } \,\
P=(x,z).
\end{equation*}
Let $L$ be the module of finite length and finite projective dimension in \cite[Theorem 5.1]{DuttaHochsterMcLaughlin}. Viewed as a $k$-vector space, $L$ has dimension $15$. Chan--Huang record its complete action table in \cite[Section~5.4]{ChanHuang}. We consider 
\[
L^\vee \cong \operatorname{Hom}_k(L,k).
\]
where $L^\vee$ is the Matlis dual of $L$ and the isomorphism follows as $L$ has finite length over $A$. By local duality, $L^\vee \cong \Ext_A^3(L, A)$ and applying $\Hom_A(-, A)$ to a finite free resolution of $L$ shows that $\pd(L^\vee)=3$. By \cite[Theorem 4.3]{DuttaHochsterMcLaughlin}, we have that $\chi(L^\vee, A/P)=-1$.\footnote{Note that the variables $x,y,z,w$ corresponds to $x_1,x_2,x_3,x_4$ in \cite{DuttaHochsterMcLaughlin} and the ideal $P=(x,z)$ corresponds to $(x_1, x_3)$ (comparing \cite[Section~5.4]{ChanHuang} and \cite[Theorem 5.1]{DuttaHochsterMcLaughlin}), thus \cite[Theorem 4.3]{DuttaHochsterMcLaughlin} under our notation says that $\chi(L^\vee, A/(x,y))=1$ and hence $\chi(L^\vee, A/(x,z))=-1$.} 

In the rest of this section we will prove that $\m=(x,y,z,w)$ annihilates $\Tor_0^A(L^\vee, A/P)$ and $\Tor_1^A(L^\vee, A/P)$ (by the acyclic lemma we know that the higher Tor's vanish). We show this via an explicit computation. With notations as in \cite[Section~5.4]{ChanHuang}, a basis of $L$ is given by 
\[
 u_1,\ldots,u_5,\quad v_1,\ldots,v_4,\quad w_1,\ldots,w_6
\]
and we let
\[
 a_i=u_i^\vee,\qquad b_j=v_j^\vee,\qquad c_\ell=w_\ell^\vee
\]
be the dual basis of $L^\vee$.  The action table of $\{x,y,z,w\}$ is the transpose of the table in \cite[Section~5.4]{ChanHuang}. The entire table is recorded as follows.

\begin{center}
\renewcommand{\arraystretch}{1.12}
\begin{tabular}{c|cccc}
\toprule
basis vector & $x$ & $y$ & $z$ & $w$\\
\midrule
$a_1$ & $b_1$ & $0$ & $b_3$ & $c_1$\\
$a_2$ & $b_2$ & $0$ & $b_4$ & $c_2$\\
$a_3$ & $0$ & $c_1$ & $c_6$ & $c_4$\\
$a_4$ & $c_5$ & $c_2$ & $0$ & $c_5$\\
$a_5$ & $c_6$ & $c_3$ & $0$ & $c_6$\\
\midrule
$b_1$ & $c_1$ & $0$ & $c_3$ & $0$\\
$b_2$ & $c_2$ & $0$ & $c_4$ & $0$\\
$b_3$ & $c_3$ & $0$ & $c_5$ & $0$\\
$b_4$ & $c_4$ & $0$ & $c_6$ & $0$\\
\midrule
$c_1$ & $0$ & $0$ & $0$ & $0$\\
$c_2$ & $0$ & $0$ & $0$ & $0$\\
$c_3$ & $0$ & $0$ & $0$ & $0$\\
$c_4$ & $0$ & $0$ & $0$ & $0$\\
$c_5$ & $0$ & $0$ & $0$ & $0$\\
$c_6$ & $0$ & $0$ & $0$ & $0$\\
\bottomrule
\end{tabular}
\end{center}
By direct inspection, 
\[
 (x,z)L^\vee=\langle b_1,\ldots,b_4,c_1,\ldots,c_6\rangle,
\]
and thus
\begin{equation*}
 \Tor_0^A(L^\vee,A/P)=L^\vee/(x,z)L^\vee
 =\langle\bar a_1,\ldots,\bar a_5\rangle
\end{equation*}
Examining the table we see that $ya_i, wa_i\in (x,z)L^\vee$ for all $i$, so $\m$ annihilates $\Tor_0^A(L^\vee,A/P)$.

We next note that $A/P$ has a two-periodic resolution
\begin{equation*}
 \cdots\longrightarrow A^2\xrightarrow{\Psi}A^2
 \xrightarrow{\Phi}A^2\xrightarrow{(x\ z)}A\longrightarrow0,
\end{equation*}
where
\[
 \Phi=\begin{pmatrix}-z&w\\x&-y\end{pmatrix},
 \qquad
 \Psi=\begin{pmatrix}y&w\\x&z\end{pmatrix},
 \qquad
 \Phi\Psi=\Psi\Phi=(xw-yz)I_2.
\]
After tensoring with $L^\vee$ and writing
\[
 d_1=(x\ z),\qquad
 d_2=\begin{pmatrix}-z&w\\x&-y\end{pmatrix},
 \qquad
 d_3=\begin{pmatrix}y&w\\x&z\end{pmatrix},
\]
we obtain
\begin{equation*}
 \cdots\longrightarrow (L^\vee)^2\xrightarrow{d_3}(L^\vee)^2
 \xrightarrow{d_2}(L^\vee)^2\xrightarrow{d_1}(L^\vee)\longrightarrow0.
\end{equation*}
We now compute the ranks of $d_1$ and $d_2$, viewed as maps of finite dimensional $k$-vector spaces. Either by examining the table or using that $\dim_kL^\vee =15$ and $\dim_k\Tor_0^A(L^\vee, A/P)=5$, we obtain that $\rank_kd_1=10$. The rank of $d_2$ can be computed by analyzing the matrix; alternatively, as we have $\chi(L^\vee, A/P)=-1$ and $\dim_k\Tor_0^A(L^\vee, A/P)=5$, it follows that $\dim_k\Tor_1^A(L^\vee, A/P)=6$ and thus $\rank_kd_2=30-10-6=14$.

We need to find a basis for $\Image(d_2)$. Use the ordered basis
\[
 (a_1,\ldots,a_5,b_1,\ldots,b_4,c_1,\ldots,c_6)
\]
in each copy of $L^\vee$ and select the 14 source columns
\begin{equation*}
 \mathcal C=\bigl((a_i,0)_{i=1}^5,(b_j,0)_{j=1}^4,(0,a_i)_{i=1}^5\bigr)
\end{equation*}
and the 14 target coordinates
\begin{equation*}
 \mathcal R=\bigl((b_3,0),(b_4,0),(c_\ell,0)_{\ell=1}^6,
 (0,c_\ell)_{\ell=1}^6\bigr).
\end{equation*}
By direct computation from the action table and the definition of $d_2$, one checks that the rows in $d_2(\mathcal C)$ correspond to $\mathcal R$ constitute a $14\times 14$ matrix whose determinant is $-1$. It follows that $d_2(\mathcal C)$ is a basis for $\Image(d_2)$.

Finally, we show that $\Tor_1^A(L^\vee, A/P)$ is annihilated by $\m$. 
We first show that the following six elements in $(L^\vee)^2$ represent a basis of $\Tor_1^A(L^\vee, A/P)$:
\begin{equation*}
\begin{aligned}
 \gamma_1&=(a_3,0),&
 \gamma_2&=(-a_5,a_3),&
 \gamma_3&=(0,a_4),\\
 \gamma_4&=(0,a_5),&
 \gamma_5&=(-a_4,b_3),&
 \gamma_6&=(-a_5,b_4).
\end{aligned}
\end{equation*}
It is easy to see that each element above is killed by $d_1$. To certify independence modulo $\Image(d_2)$, take the 14 columns $d_2(\mathcal C)$ from $ \mathcal R$ and append $\gamma_1,\ldots,\gamma_6$.  On the 20 target coordinates
\begin{equation*}
\begin{aligned}
 \mathcal R'=\bigl(& (a_3,0),(a_4,0),(a_5,0),(b_3,0),(b_4,0),
 (c_\ell,0)_{\ell=1}^6,\\
 & (0,a_3),(0,a_4),(0,a_5),(0,c_\ell)_{\ell=1}^6\bigr),
\end{aligned}
\end{equation*}
the resulting $20\times20$ determinant is $-1$.  Since $\dim_k\ker (d_1)=20$, these $20$ vectors form a basis of $\ker(d_1)$, proving the assertion. It remains to check that $y$ and $w$ kill these six cycles.  For each $i$, the following table gives elements $\rho_i,\sigma_i\in (L^\vee)^2$ satisfying
\begin{equation*}
 d_2(\rho_i)=y\gamma_i,
 \,\ \text{ and } \,\ d_2(\sigma_i)=w\gamma_i.
\end{equation*}
Recall that $d_2(s,t)=(-zs+wt,\,xs-yt)$.

\vspace{0.3em}

\begin{center}
\renewcommand{\arraystretch}{1.22}
\begin{tabular}{c|>{\raggedright\arraybackslash}p{0.39\textwidth}|>{\raggedright\arraybackslash}p{0.39\textwidth}}
\toprule
$i$ & $\rho_i$ with $d_2(\rho_i)=y\gamma_i$ & $\sigma_i$ with $d_2(\sigma_i)=w\gamma_i$\\
\midrule
$1$ & $(0,a_1)$ & $(-a_3-b_2-b_3,-a_4-a_5)$\\
$2$ & $(b_1,0)$ & $(b_4,0)$\\
$3$ & $(-a_3-b_3,-a_4-a_5)$ & $(a_4,0)$\\
$4$ & $(-a_3,-a_5)$ & $(a_5,0)$\\
$5$ & $(0,-a_2)$ & $(a_3+b_3,a_5)$\\
$6$ & $(a_3+b_1+b_2+b_3,a_3+a_4+a_5)$ & $(a_3,0)$\\
\bottomrule
\end{tabular}
\end{center}

\vspace{0.3em}

\noindent These follow from direct computations from the action table and the definition of $d_2$. Therefore $y$ and $w$ (and hence $\m$) annihilate $\Tor_1^A(L^\vee, A/P)$. We summarize what we have shown regarding $L^\vee$:
\begin{equation*}
 \Tor_i^A(L^\vee,A/P)\cong
 \begin{cases}
 k^5,&i=0,\\
 k^6,&i=1,\\
 0,&i\ge2,
 \end{cases}
 \qquad\text{as $A$-modules.}
\end{equation*}

Let $M_A = L^\vee \oplus A/(w, y+z)$ and $C= A/P\cong k[[y,w]]$. Since $w, y+z$ is a regular sequence on $A$, we have $M_A$ is a finitely generated $A$-module of finite projective dimension with $\dim(M_A)=1$. Moreover, since the images of $w, y+z$ in $C$ is a regular system of parameters, we have $(A/(w, y+z))\otimes_A^\mathbb{L}C \cong k$ in $D(C)$. Thus, setting
$$G:=M_A \otimes_A^\mathbb{L}C,$$
we know that $G$ is a perfect complex in $D(C)$ with $H^0(G)\cong H^{-1}(G)\cong k^6$ as $C$-modules (i.e., the two nontrivial cohomologies of $G$ are both annihilated by $\m_C=(y,w)$).

\subsection{A lifting lemma}
\label{sec: lifting lemma}
Continuing the notation as in the last paragraph in the previous section, we want to find a finite length module over $C[[u]]$ so that the perfect complex $G$ is quasi-isomorphic to the module quotient by $u$ in the derived sense. In general, we do not know a satisfactory criterion for the existence of such modules. However, in our scenario, the cohomologies of $G$ are annihilated by $\m$ and we can turn the question into an explicit linear algebra computation.

\begin{lemma}
\label{lem: key lifting lemma}
Let $C=k[[y,w]]$ and $B=C[[u]]$. Suppose $G$ is a perfect complex in $D(C)$ so that $H^0(G)\cong H^{-1}(G)\cong k^n$ for some $n>1$ and $H^i(G)=0$ for $i\neq 0,-1$. Then there exists a finite length $B$-module $M_B$ so that $M_B\otimes_B^\mathbb{L}C\cong G$ in $D(C)$.
\end{lemma}
\begin{proof}
We first observe that by assumption we have an exact triangle $H^{-1}(G)[1]\to G \to H^0(G)\xrightarrow{+1}$. Thus $G$ is determined by a class in $\Ext_C^1(H^0(G), H^{-1}(G)[1])\cong \Ext_C^2(k^n, k^n)$. Therefore, it suffices to show that for any $\eta \in \Ext_C^2(k^n, k^n)$, we can find a finite length $B$-module $N$ so that $\eta$ is represented by the exact sequence: 
\begin{equation}
\label{eqn: Yoneda class}
0 \to k^n \to N \xrightarrow{u} N \to k^n \to 0 \tag{$\dagger$}
\end{equation}
where we view this as a sequence of finite length $C$-modules and we view $u$ as a $C$-linear map. 

We will construct $N$ as follows. Let $V_0 \cong V_1\cong k^n$ and let $N=V_1\oplus V_0$ as $k$-vector spaces. Let $u,y,w$ act on $N$ by sending $V_1$ to $0$ and assigning for every $v\in V_0$
\begin{equation*}
u(v)=Uv,\qquad y(v)=UK_yv,\qquad w(v)=UK_wv,
\end{equation*}
with all output values regarded as elements of $V_1$, and $U,K_y, K_w$ are $n\times n$ matrices with entries in $k$ so that $U$ is invertible.  Note that every product of two elements among $u,y,w$ acts as zero. The three operators therefore commute and define a finite length $B$-module with $\m_B^2N=0$. Moreover, since $U$ is invertible, we have $N/uN\cong V_0 \cong k^n$ and $\ker(N\xrightarrow{u} N)\cong V_1\cong k^n$. To determine the Yoneda class of (\ref{eqn: Yoneda class}) in $\Ext_C^2(k^n, k^n)$, we consider the following diagram:
\[
\xymatrix{
0 \ar[r] & C^n  \ar[r]^-{(-w, y)^T} \ar@{.>}[d]^\gamma & C^n \oplus C^n \ar[r]^-{(y,w)} \ar@{.>}[d]^\beta & C^n \ar[r] \ar@{.>}[d]^\alpha & k^n \ar[r] \ar@{=}[d] & 0\\ 
0 \ar[r] & k^n  \ar[r]  & V_0 \oplus V_1 \ar[r]^-{u}  & V_0 \oplus V_1 \ar[r]  & k^n \ar[r]  & 0\\
}
\]
where the maps $\alpha,\beta,\gamma$ are arbitrary lifts to make the diagram commutative. By definition of the $u,y,w$ actions on $V_0 \oplus V_1$, it is straightforward to check that one can define 
\begin{align*}
  \alpha(e_i) &= (\overline{e_i}, 0) \\  
\beta(e_i,e_j) & = (K_y\overline{e_i}+ K_w\overline{e_j}, 0)\\
\gamma(e_i) & = U(K_yK_w-K_wK_y)\overline{e_i}
\end{align*}
where $e_i$ (resp., $\overline{e_i}$) denotes the $i$-th basis vector in $C^n$ (resp., $k^n$).

Now the Yoneda class of (\ref{eqn: Yoneda class}) is represented by the map $\gamma$. Unwinding definitions, this is given by 
$$U(K_yK_w-K_wK_y) \cdot \xi \in \Ext_C^2(k,k)^{\oplus n^2} \cong \Ext_C^2(k^n, k^n)$$
where $\xi$ denotes the canonical class in $\Ext_C^2(k, k)\cong k$. Thus, to show that every class in $\Ext_C^2(k^n, k^n)$ can be represented by an exact sequence as in (\ref{eqn: Yoneda class}), it suffices to show that for every $n\times n$ matrix $K\in \text{M}_{n\times n}(k)$, we can find $U, K_1, K_2 \in \text{M}_{n\times n}(k)$ with $U$ invertible so that 
$$K = U(K_1K_2-K_2K_1).$$
This is elementary: first observe that we can choose $U$ so that $\text{tr}(U^{-1}K)=0$, indeed, if $U_1, U_2$ are invertible matrices so that
\[K=
 U_1\begin{pmatrix} I_r &0\\0&0\end{pmatrix}U_2
\]
where $r=\rank K$, then we can take $U=U_1PU_2$ where $P$ is any fixed-point-free permutation matrix (this is the only place we used $n>1$), and then we use the fact that every trace zero matrix can be written as the commutator of two matrices, for example see \cite{AlbertMuckenhouptMatricesTraceZero}.
\end{proof}

\begin{remark}
Lemma~\ref{lem: key lifting lemma} is not true if $n=1$. In particular, if $G$ corresponds to the nontrivial element in $\Ext_C^2(k, k)$ then we claim there is no $B$-module $N$ so that $N\otimes_B^\mathbb{L}C\cong G$. For suppose such $N$ exists, then viewed as a $k[[u]]$-module, $N\cong k[[u]]/(u^m)$ for some $m\geq 1$ by Nakayama's lemma. Let $y,w$ act on $N$ via multiplication by $f(u),g(u)$. Consider the diagram:
\[
\xymatrix{
0 \ar[r] & C  \ar[r]^-{(-w, y)^T} \ar@{.>}[d]^\gamma & C^2 \ar[r]^-{(y,w)} \ar@{.>}[d]^\beta & C \ar[r] \ar@{.>}[d]^\alpha & k \ar[r] \ar@{=}[d] & 0\\ 
0 \ar[r] & k  \ar[r]  & N \ar[r]^-{u}  & N \ar[r]  & k \ar[r]  & 0.\\
}
\]
where $\alpha, \beta,\gamma$ are lifts to make the diagram commutative. Without loss of generality, we assume that $\alpha(1)=1$ and $\gamma(1)=\lambda$. If $\beta(e_1)=f_1(u)$ and $\beta(e_2)=g_1(u)$, then chasing the diagram we find that 
\begin{align*}
 f(u) &= uf_1(u) \\  
 g(u) & = ug_1(u)\\
\lambda u^{m-1} & = - g(u)f_1(u) + f(u)g_1(u).
\end{align*}
But these equations force $\lambda=0$, contradicting that $G\cong N\otimes_B^\mathbb{L}C$ corresponds to the nontrivial class in $\Ext_C^2(k, k)$.
\end{remark}

\begin{corollary}
\label{cor: key lifting lemma two variable}
Let $C=k[[y,w]]$ and $B'=C[[u,v]]$. Suppose $G$ is a perfect complex in $D(C)$ so that $H^0(G)\cong H^{-1}(G)\cong k^n$ for some $n>1$ and $H^i(G)=0$ for $i\neq 0,-1$. Then there exists a finitely generated $B'$-module $M_{B'}$ with $\dim(M_{B'})=1$ so that $M_{B'}\otimes_{B'}^\mathbb{L}C\cong G$ in $D(C)$.
\end{corollary}
\begin{proof}
This follows immediately from Lemma~\ref{lem: key lifting lemma} by taking $M_{B'}=M_B\otimes_B B'$. 
\end{proof}

\begin{remark}
In our case of interest, there is an alternative way to see that there is a finite length $B$-module $M_B$ so that $M_B\otimes^\mathbb{L}_BC\cong G$, where $G$ is the complex constructed at the end of Section~\ref{sec: DHM}. By \cite[Remark (5.3)]{DuttaHochsterMcLaughlin}, a minimal free resolution of $L^\vee$ has the form $$0 \to A^6 \to A^{17} \to A^{16} \to A^5\to L^\vee \to 0.$$ Thus after taking direct sum with a Koszul resolution of $A/(w, y+z)$ and killing $P$, a minimal free resolution of $G$ has the form $0 \to C^6 \to C^{18} \to C^{18} \to C^6\to G \to 0$. We also know from Section~\ref{sec: DHM} that $H^{-1}(G)\cong H^0(G)\cong k^6$. This implies that we must have $G\cong k^6 \oplus k^6[1]$ in $D(C)$: a minimal free resolution of the latter has the same form $0 \to C^6 \to C^{18} \to C^{18} \to C^6\to G \to 0$, and if $G$ corresponds to a nontrivial class in $\Ext_C^2(k^6, k^6)$, then by examining the mapping cone of $k^6 \to k^6[2]$, that class leads to units in one of the differentials in the mapping cone and thus the Betti numbers of $G$ cannot be $(6, 18, 18, 6)$. Once we know $G\cong k^6 \oplus k^6[1]$, it is easy to see that we can simply take $M_B=k^6$. This argument bypasses Lemma~\ref{lem: key lifting lemma}, but requires more knowledge on the module constructed in \cite{DuttaHochsterMcLaughlin}. 
\end{remark}

\subsection{Counterexample to the dimension inequality}
\label{sec: counterexample dimension inequality}
We now construct a counterexample to Conjecture~\ref{conj: Peskine--Szpiro} $(i)$. We continue to use the notation at the end of Section~\ref{sec: DHM}. Let $B=C[[u]]$ and let $R = A \times_CB$.
We can write down an explicit presentation of $R$:
$$R = \frac{k[[x,y,z,w,u]]}{(xw - yz, ux, uz)}$$
and it is easy to see that $A=R/(u)$ and $B=R/(x,z)$ are the two components of $R$.
Recall that we have a finitely generated $A$-module $M_A$ of finite projective dimension and, by Lemma~\ref{lem: key lifting lemma}, a finitely generated $B$-module $M_B$ of finite length so that 
$$M_A \otimes^\mathbb{L}_A C \cong G \cong M_B\otimes_B^\mathbb{L}C.$$
By Milnor patching for perfect complexes (\cite[Chapter I, Theorem 2.7]{WeibelKbook} or \cite[Section 5]{BenBassatBlockMilnorDescent})\footnote{More concretely, we can take minimal free resolutions of $M_A$ and $M_B$, call them $F_A$ and $F_B$. Note that $F_A \otimes_AC \cong F_B\otimes_BC$ as chain complexes since they are both minimal over $C$ and both quasi-isomorphic to $G$. Abusing notations, we use $G$ to denote this complex. Then we simply define $F$ term by term via $F_n = (F_A)_n \times_{G_n} (F_B)_n$.}, there is a perfect complex $F$ over $R$ so that 
$$F \otimes^\mathbb{L}_R A \cong M_A \,\ \text{ and } \,\ F \otimes^\mathbb{L}_R B \cong M_B$$
and an exact triangle:
$$F \to M_A\oplus M_B \to G \xrightarrow{+1}.$$
By examining the long exact sequence on cohomology, we find that $H^i(F)=0$ unless $i=0$. It follows that $M:= H^0(F)$ is a finitely generated $R$-module of finite projective dimension. Note that we have $\dim(M)=1$ since $\dim(M_A)=1$ (and the cohomologies of $M_B$ and $G$ have finite length). Finally, we have $M \otimes_R B\cong M_B$ which has finite length, but 
$$\dim(M) + \dim(B) = 1 + 3 > 3 = \dim(R).$$

\subsection{Counterexample to the grade conjecture}
\label{sec: counterexample grade conjecture}
In this section we construct a counterexample to Conjecture~\ref{conj: Peskine--Szpiro} $(iii)$. We continue to use the notation at the end of Section~\ref{sec: DHM}. Let $B'=C[[u,v]]$ and let $R' = A \times_CB'$.
We can write down an explicit presentation of $R'$:
$$R' = \frac{k[[x,y,z,w,u,v]]}{(xw - yz, ux, uz, vx, vz)}$$
and it is easy to see that $A=R'/(u, v)$ and $B'=R'/(x,z)$ are the two components of $R'$. Note that $\dim(A)=3$ and $\dim(B')=4$, so $\dim(R')=4$ and $R'$ is not equidimensional.

Consider the finitely generated $A$-module $M_A$ of finite projective dimension and, by Corollary~\ref{cor: key lifting lemma two variable}, the finitely generated $B'$-module $M_{B'}$ so that 
$$M_A \otimes^\mathbb{L}_A C \cong G \cong M_{B'}\otimes_{B'}^\mathbb{L}C.$$
By Milnor patching for perfect complexes as above, there is a perfect complex $F'$ over $R'$ so that 
$$F' \otimes^\mathbb{L}_{R'} A \cong M_A \,\ \text{ and } \,\ F' \otimes^\mathbb{L}_{R'} B' \cong M_{B'}$$
and an exact triangle:
$$F' \to M_A\oplus M_{B'} \to G \xrightarrow{+1}.$$
By examining the long exact sequence on cohomology, we find that $H^i(F')=0$ unless $i=0$.  It follows that $M':= H^0(F')$ is a finitely generated $R'$-module of finite projective dimension. Note that we have $\dim(M')=1$ since $\dim(M_A)=\dim(M_{B'})=1$ (and the cohomologies of $G$ have finite length).

We now show that $M'$ violates the grade conjecture. Since $\dim(R') - \dim(M')= 4-1=3$. It is enough to show that $\grade(M')\leq 2$. Since $\grade(M') := \depth_{\Ann_{R'}(M')}R' \leq \height(\Ann_{R'}(M'))$, it suffices to show that $\height(\Ann_{R'}(M'))\leq 2$. Note that 
$$\Ann_{R'}(M') \cdot A \subseteq \Ann_{A}(M'\otimes_{R'}A) = \Ann_A(M_A) \subseteq (w, y+z)A \subseteq (y,z,w)A.$$
It follows that $\Ann_{R'}(M') \subseteq (y,z,w,u,v)R' =:Q$. Thus we have 
$$\height(\Ann_{R'}(M')) \leq \dim(R'_Q) =\dim(A_{(y,z,w)A})=2$$
where the first equality above used the fact that $x\notin Q$ and hence localizing at $Q$ kills the other component $B'$ of $R'$.


\newpage
\appendix
\section{Negative intersection multiplicities revisited}
\label{app: negative intersection}
\counterwithin{equation}{section}
\counterwithin{theorem}{section}
\setcounter{theorem}{0}
\setcounter{equation}{0}
In this appendix, we give a streamlined proof explaining the existence of the Dutta--Hochster--McLaughlin's example \cite{DuttaHochsterMcLaughlin}, following \cite{RobertsSrinivasModulesoffinitelengthandfiniteprojectivedimension,KuranoNumericalequivalenceonChowgroupsoflocalrings} (all unexplained terminology in what follows can be found in these references).

Let $k$ be a field and let $R$ be a standard graded $k$-algebra with homogeneous maximal ideal $\m$. Let $\pi$: $Y=\Proj (R[\m t])\to \Spec(R)$ be the blowup of $\Spec(R)$ at $\m$. The exceptional divisor can be identified with $X=\Proj(R)$, and $\sO_Y(-X)$ restricts to the ample line bundle $L=\sO_X(1)$. We have a natural isomorphism $Y\cong {\Spec}_X(\oplus_{i\geq 0}L^i)$ and the inclusion $X\hookrightarrow Y$ is the zero section. We consider two pairings induced by tensor product:
\[\xymatrix{
G_0(R) \times K_0^{\m}(R)  \ar[r]  \ar@<4ex>[d]^{\pi^*} & G_0^{\m}(R) \ar[r]^-\cong & \mathbb{Z} \\
G_0(Y) \times K_0^X(Y) \ar@<4ex>[u]^{R\pi_*} \ar[r]  & G_0^X(Y) \ar[u]^{R\pi_*},
}
\]
where the isomorphism above is induced by taking the Euler characteristic. The two pairings are compatible via the projection formula:
$$R\pi_*(\mathcal{G}\otimes \pi^*\mathcal{F})=R\pi_*\mathcal{G} \otimes \mathcal{F}$$
for $\mathcal{G}\in G_0(Y)$ and $\mathcal{F}\in K_0^{\m}(R)$.

In what follows, we assume $X$ is nonsingular, thus $Y$ is also nonsingular. In this case, we have natural isomorphisms 
$$G_0(X)\cong K_0(X), \,\ G_0(Y)\cong K_0(Y), \,\ G_0(X)\cong G_0(Y)$$ 
where the last isomorphism is induced by the pull back $p^*$, where $p$: $Y\cong{\Spec}_X(\oplus_{i\geq 0}L^i) \to X$ is the natural map. Furthermore, we also have isomorphisms:
$$K_0(X)\cong G_0(X) \cong G_0^X(Y) \cong K_0^X(Y).$$
Putting these together, we rewrite the two pairings as:
\[\xymatrix{
G_0(R) \times K_0^{\m}(R)  \ar[r]  \ar@<4ex>[d]^{\beta} &  \mathbb{Z} \\
K_0(X) \times K_0(X) \ar@<4ex>[u]^{\alpha} \ar[r]  & \mathbb{Z}, \ar@{=}[u]
}
\]
where $\alpha$, $\beta$ denote the maps induced by $R\pi_*$ and $\pi^*$ (composed with the isomorphisms), and the bottom pairing is the tensor product followed by taking derived global sections, and then taking the Euler characteristic. The two parings are compatible via 
\begin{equation}
\label{eqn: compatibility via projection formula}
\langle\alpha(\mathcal{G}), \mathcal{F}\rangle= \langle\mathcal{G}, \beta(\mathcal{F})\rangle
\end{equation}
for $\mathcal{G}\in K_0(X)$ and $\mathcal{F}\in K_0^{\m}(R)$.


Unwinding definitions and the isomorphisms, we have $\alpha(\mathcal{G}) = \oplus_{i\geq 0} R\Gamma(X, \mathcal{G}(i))$. If we consider the induced map $\alpha$: $K_0(X)_\mathbb{Q}\to G_0(R)_\mathbb{Q}$, then since $[\oplus_{i\geq 0} R\Gamma(X, \mathcal{G}(i))] =[\oplus_{i\in\mathbb{Z}}\Gamma(X, \mathcal{G}(i))]$ in $G_0(R)_\mathbb{Q}$ (the terms we ignored are supported at $\m$ and so become $0$ in $G_0(R)_\mathbb{Q}$), the map $\alpha$ sits inside the following exact sequence:
\begin{equation}
\label{eqn: relating G_0 with affine cone}
K_0(X)_\mathbb{Q}\xrightarrow{\cdot h}K_0(X)_\mathbb{Q}\xrightarrow{\alpha}G_0(R)_\mathbb{Q}\to 0,
\end{equation}
where $h=[\sO_X]-[\sO_X(-1)]$. This is essentially a reformulation of the standard exact sequence $G_0(X)\rightarrow G_0(Y) \to G_0(Y\backslash X)\to 0$, after identifying $G_0(X)$, $G_0(Y)$ with $K_0(X)$ and noting that $G_0(Y\backslash X)_\mathbb{Q}=G_0(\Spec(R)\backslash \{\m\})_\mathbb{Q}=G_0(R)_\mathbb{Q}$. 
We thus have an induced diagram:
\[\xymatrix{
G_0(R)_\mathbb{Q} \times K_0^\m(R)_\mathbb{Q}  \ar[r]  \ar@<4ex>[d]^{\beta} &  \mathbb{Q} \\
K_0(X)_\mathbb{Q} \times K_0(X)_\mathbb{Q} \ar@{->>}@<4ex>[u]^{\alpha} \ar[r]  & \mathbb{Q}. \ar@{=}[u]
}
\]
The compatibility of the two pairings (\ref{eqn: compatibility via projection formula}) shows that we have an induced diagram modulo numerical equivalence:
\begin{equation}
\label{eqn: diagram mod num equiv}
\xymatrix{
\overline{G_0(R)}_\mathbb{Q} \times \overline{K_0^{\m}(R)}_\mathbb{Q}  \ar[r]  \ar@{^{(}->}@<4ex>[d]^{\overline{\beta}} &  \mathbb{Q} \\
\overline{K_0(X)}_\mathbb{Q} \times \overline{K_0(X)}_\mathbb{Q} \ar@{->>}@<4ex>[u]^{\overline{\alpha}} \ar[r]  & \mathbb{Q}, \ar@{=}[u]
}
\end{equation}
where the injectivity of $\overline{\beta}$ simply follows from the surjectivity of $\overline{\alpha}$. By \cite[Theorem 3.1]{KuranoNumericalequivalenceonChowgroupsoflocalrings}, all terms in (\ref{eqn: diagram mod num equiv}) are finite dimensional $\mathbb{Q}$-vector spaces and the two pairings are perfect pairings.

\begin{proposition}[{\cite{RobertsSrinivasModulesoffinitelengthandfiniteprojectivedimension,KuranoNumericalequivalenceonChowgroupsoflocalrings}}]
\label{prop: RobertSrinivasKuranoKey}
With notations as above, we have 
$$\overline{G_0(R)}_\mathbb{Q} \cong \coker\big(\overline{K_0(X)}_\mathbb{Q}\xrightarrow{\cdot h}\overline{K_0(X)}_\mathbb{Q}\big)$$ 
if and only if 
$$\overline{K_0^{\m}(R)}_\mathbb{Q} \cong \ker\big(\overline{K_0(X)}_\mathbb{Q}\xrightarrow{\cdot h}\overline{K_0(X)}_\mathbb{Q}\big).$$
\end{proposition}
\begin{proof}
We will prove that $\overline{\alpha},\overline{\beta}$ factor through $\coker(h),\ker(h)$ respectively as follows:
\begin{equation}
\label{eqn: diagram mod num equiv factorization}
\xymatrix{
\overline{G_0(R)}_\mathbb{Q} \times \overline{K_0^{\m}(R)}_\mathbb{Q}  \ar[r]  \ar@{^{(}->}@<4ex>[d]^{b} &  \mathbb{Q} \\
\coker(h) \,\  \,\ \ker(h) \ar@{->>}@<4ex>[u]^{a}  \ar@{^{(}->}@<4ex>[d] &   \\
\overline{K_0(X)}_\mathbb{Q} \times \overline{K_0(X)}_\mathbb{Q} \ar@{->>}@<4ex>[u] \ar[r]  & \mathbb{Q}. \ar@{=}[uu]
}
\end{equation}
Here, the factorization of $\overline{\alpha}$ follows from (\ref{eqn: relating G_0 with affine cone}) after modulo numerical equivalence. For the factorization of $\overline{\beta}$, we invoke the following localization sequence of $K$-theory \cite{ThomasonTrobaughKtheory}:
\[
\xymatrix{
\cdots \ar[r] & K_1(Y\backslash X)_\mathbb{Q} \ar[r] & K_0^X(Y)_\mathbb{Q} \ar[r] & K_0(Y)_\mathbb{Q} \ar[r] & \cdots \\
\cdots \ar[r] & K_1(\Spec(R)\backslash\{\m\})_\mathbb{Q} \ar@{=}[u] \ar[r] & K_0^{\m}(R)_\mathbb{Q} \ar[r] \ar[u]^{\pi^*} & K_0(R)_\mathbb{Q} \ar[r] \ar[u] & \cdots.
}
\]
Since $K_0^{\m}(R)_\mathbb{Q} \to K_0(R)_\mathbb{Q}\cong\mathbb{Q}$ is the zero map, chasing the diagram we obtain an exact sequence
$$K_0^{\m}(R)_\mathbb{Q} \to K_0^X(Y)_\mathbb{Q} \to K_0(Y)_\mathbb{Q}.$$
Unwinding the definition of $\beta$, this sequence can be identified with 
\begin{equation}
\label{eqn: key factorization from K-theory}
K_0^{\m}(R)_\mathbb{Q} \xrightarrow{\beta} K_0(X)_\mathbb{Q} \xrightarrow{h} K_0(X)_\mathbb{Q}.
\end{equation}
Thus $\beta$ (and hence $\overline{\beta}$) factors through $\ker(h)$ as wanted. By counting dimensions over $\mathbb{Q}$ in (\ref{eqn: diagram mod num equiv factorization}) and using that both pairings are perfect pairings, it is straightforward to see that $a$ is bijective if and only if $b$ is bijective.
\end{proof}

\begin{corollary}
\label{cor: CH equals CH mod num equiv}
With notations as above, if $K_0(X)_\mathbb{Q}\cong \overline{K_0(X)}_\mathbb{Q}$, then $G_0(R)_\mathbb{Q}\cong\overline{G_0(R)}_\mathbb{Q}$. As a consequence, if $\emph{CH}_*(X)_\mathbb{Q}=\overline{\emph{CH}_*(X)}_\mathbb{Q}$, then $A_*(R)_\mathbb{Q}\cong \overline{A_*(R)}_\mathbb{Q}$ and $A_i(R)_\mathbb{Q}\cong \overline{A_i(R)}_\mathbb{Q}$ for all $i$.
\end{corollary}
\begin{proof}
We first observe that $\overline{K_0^{\m}(R)}_\mathbb{Q} \cong \ker\big(\overline{K_0(X)}_\mathbb{Q}\xrightarrow{\cdot h}\overline{K_0(X)}_\mathbb{Q}\big)$: injectivity always holds as in (\ref{eqn: diagram mod num equiv factorization}) and surjectivity follows from (\ref{eqn: key factorization from K-theory}) and the assumption that $K_0(X)_\mathbb{Q}\cong \overline{K_0(X)}_\mathbb{Q}$. Proposition~\ref{prop: RobertSrinivasKuranoKey} implies that $\overline{G_0(R)}_\mathbb{Q} \cong \coker\big(\overline{K_0(X)}_\mathbb{Q}\xrightarrow{\cdot h}\overline{K_0(X)}_\mathbb{Q}\big)$. The first statement follows from (\ref{eqn: relating G_0 with affine cone}) by using the assumption $K_0(X)_\mathbb{Q}\cong \overline{K_0(X)}_\mathbb{Q}$ again.

The second statement follows from the first by invoking the (singular) Riemann--Roch to identify $K_0(X)_\mathbb{Q}$ with $\text{CH}_*(X)_\mathbb{Q}$ and $G_0(R)_\mathbb{Q}$ with $A_*(R)_\mathbb{Q}$, see \cite[Corollary 18.3.2]{FultonIntersectionTheoryOriginal} (these identifications are compatible with numerical equivalence, see \cite[Section 2]{KuranoNumericalequivalenceonChowgroupsoflocalrings}). The last assertion follows from \cite[Proposition 2.4]{KuranoNumericalequivalenceonChowgroupsoflocalrings}.
\end{proof}

We can now explain the existence of modules of finite length and finite projective dimension with negative intersection multiplicity over the ring $k[x_1,x_2, x_3, x_4]/(x_1x_4-x_2x_3)$. In fact, one can construct many examples as in \cite[Examples 7.8 and 7.9]{KuranoNumericalequivalenceonChowgroupsoflocalrings}.

\begin{corollary}[{\cite{DuttaHochsterMcLaughlin}}]
Suppose $R=k[x_1,x_2, x_3, x_4]/(x_1x_4-x_2x_3)$ and $P=(x_1,x_2)$. Then there exists an $R$-module $M$ of finite length and finite projective dimension with $\chi(M,R/P)< 0$.
\end{corollary}
\begin{proof}
In this case $X=\Proj(R)=\mathbb{P}^1\times \mathbb{P}^1$, so $\text{CH}_*(X)_\mathbb{Q}=\overline{\text{CH}_*(X)}_\mathbb{Q}=\mathbb{Q}^2\oplus \mathbb{Q}^2$. Corollary~\ref{cor: CH equals CH mod num equiv} then implies that 
$$\overline{A_2(R)}_\mathbb{Q}\cong A_2(R)_\mathbb{Q} \cong \text{Cl}(R)_\mathbb{Q} \neq 0.$$ 
In particular, the image of the class $[R/P]$ (which is a generator of the class group) is nonzero in $\overline{A_2(R)}_\mathbb{Q}$ and thus nonzero in $\overline{G_0(R)}_\mathbb{Q}$. Unwinding the definitions, this implies that there exists a perfect complex $\mathcal{F}$ supported at $\{\m\}=(x_1,x_2, x_3, x_4)$ such that $\chi(\mathcal{F}, R/P)\neq 0$. On the other hand, the Grothendieck group of perfect complexes supported at $\{\m\}$ is generated by modules of finite length and finite projective dimension by \cite[Proposition 2]{RobertsSrinivasModulesoffinitelengthandfiniteprojectivedimension}. Thus there exists such an $M$ with $\chi(M,R/P)\neq 0$. Replacing $[R/P]$ with $[R/Q]$ (where $Q=(x_1,x_3)$) if necessary and using symmetry (or alternatively, replacing $M$ by $M^\vee$), we can arrange that $\chi(M, R/P)<0$. 
\end{proof}

\begin{remark}
The counterexamples in Section 2 crucially rely on the example in \cite{DuttaHochsterMcLaughlin}. Although the negativity of intersection multiplicity is necessary for the construction, it seems also important that the relevant $\Tor$ modules are actually annihilated by the maximal ideal (see Section~\ref{sec: lifting lemma}). We do not yet have a theoretical (or less computational) explanation why this happens in the Dutta--Hochster--McLaughlin's example.
\end{remark}

\newpage
\section{Some positive results}
\label{app: positive results}
\counterwithin{equation}{section}
\counterwithin{theorem}{section}
\setcounter{theorem}{0}
\setcounter{equation}{0}
In this appendix, we collect some partial positive results towards Peskine--Szpiro's Conjecture~\ref{conj: Peskine--Szpiro}. These results are well-known to experts and are essentially contained in \cite[Chapter 6]{RobertsBook}. We first recall some backgound. For an arbitrary finitely generated module $M$ over a Noetherian local ring $(R,\m)$, we always have the following inequalities:
\begin{equation}
\label{eqn: inequalities grade codim pd}
\grade(M) \leq \codim(M) \leq \dim(R)-\dim(M) \leq \pd(M).
\end{equation}
Here $\grade(M) := \depth_{\Ann_R(M)}R$; $\codim(M) := \height(\Ann_R(M))$; the first two inequalities in (\ref{eqn: inequalities grade codim pd}) are obvious and the last inequality follows from the new intersection theorem \cite{RobertsIntersectionTheorem}, which says that if a non-exact finite free complex $F_\bullet$ is supported only at $\{\m\}$ (i.e., $F_\bullet$ has finite length cohomologies), then $\text{length}(F_\bullet)\geq \dim(R)$. 

The module $M$ is called perfect if all the inequalities in (\ref{eqn: inequalities grade codim pd}) are equalities. Peskine--Szpiro's intersection theorem (which is a consequence of the new intersection theorem in \cite{RobertsIntersectionTheorem}) says that, under the same notations as in Conjecture~\ref{conj: Peskine--Szpiro}, we have 
$$\dim(N) \leq \pd(M).$$
In particular, the intersection theorem implies that the strong intersection conjecture (and thus the dimension inequality and the grade conjecture) hold when $M$ is perfect.

It was pointed out by Roberts in \cite[Page 117]{RobertsBook} that the grade conjecture follows from generalized vanishing of Serre intersection multiplicity (i.e., under the same notations as in Conjecture~\ref{conj: Peskine--Szpiro}, we have $\chi(M,N)=0$ if $\dim(M)+\dim(N)<\dim(R)$),\footnote{Note that, however, this generalized vanishing is false even over hypersurfaces by \cite{DuttaHochsterMcLaughlin}.} who attributed this result to Peskine--Szpiro \cite{PeskineSzpiroSyzygiesMultiplicities}. Both \cite{RobertsBook} and \cite{PeskineSzpiroSyzygiesMultiplicities} only provide full details of this result in the $\mathbb{N}$-graded setting. Here we include a complete argument, closely following \cite[Theorem 6.3.4]{RobertsBook}. Henceforth, we will repeatedly use the fact that if $R$ admits a module $M$ of finite length and finite projective dimension, then $R$ is Cohen--Macaulay (this is a straightforward consequence of the new intersection theorem).

\begin{theorem}
\label{thm: vanishing implies grade}
Let $R$ be a Noetherian complete local ring and let $M$ be a finitely generated $R$-module of finite projective dimension. Suppose for every ideal $J\subseteq R$ such that $M \otimes_R R/J$ has finite length and $\dim(M)+ \dim(R/J) <\dim(R)$, we have $\chi(M, R/J)=0$. Then $$\grade(M) = \dim(R) - \dim(M).$$
\end{theorem}
\begin{proof}
Since $M$ has finite projective dimension, we know that $\grade(M)=\codim(M)$ by \cite[Proposition (2.5)]{AvramovFoxbyCMdefect}. Hence it is enough to show that 
$$\codim(M) = \dim(R) - \dim(M).$$
Without loss of generality, we may assume that $R$ is not equidimensional (as otherwise the grade conjecture holds by \cite[Proposition (A.2) and Corollary (A.3)]{AvramovFoxbyCMdefect}). Let $\mathfrak{q}_1,\dots,\mathfrak{q}_t$ be the minimal primes of $R$ so that $\dim(R/\mathfrak{q}_i)<\dim(R)$.

Let $Q_1,\dots,Q_s$ be the minimal primes in $\Supp(M)$ such that $\height(Q_i)<\dim(R)-\dim(M)$. If no such $Q$'s exist (i.e., $s=0$), then $\codim(M)\geq \dim(R) -\dim(M)$ and thus we must have equality (as the other inequality is obvious). Thus without loss of generality, we may assume $s\geq 1$. Next, we note that each minimal prime of $R_{Q_i}$ is of the form $\q_j(R_{Q_i})$ for some $j$, since if $Q_i$ contains a minimal prime $\q$ with $\dim(R/\q)=\dim(R)$, then $$\height(Q_i)=\dim(R)-\dim(R/Q_i)\geq \dim(R)-\dim(R/I)$$ which contradicts our choice of $Q_i$.

Let $\fra$ be any ideal of $R$ whose minimal primes are precisely those $\q_1,\dots,\q_t$ that are contained in $Q_i$ for some $i$. We make two observations.

\begin{claim}
\label{clm: dimension inequality}
$\dim(R/\fra) -\dim(M/\fra M) <\dim(R)- \dim(M)$.
\end{claim}
\begin{proof}[Proof of Claim]
After re-arranging the $Q_i$'s and $\q_j$'s, we may assume that $\dim(R/\fra) =\dim(R/\q_1)$ and that $\q_1$ is contained in $Q_1$. Since $Q_1\in\Supp(M/\fra M)$ by our assumption, we have 
\[
  \dim(R/\fra) -\dim(M/\fra M)  \leq \dim(R/\q_1) - \dim(R/Q_1) = \height(Q_1) < \dim(R)-\dim(M).
\]
where for the equality above, we use the fact that $R_{Q_i}$ is Cohen--Macaulay, in particular equidimensional, for each $1\leq i\leq s$ (since $M_{Q_i}$ is a module of finite length and finite projective dimension over $R_{Q_i}$) and that $\q_1$ is a minimal prime that is contained in $Q_1$ by our choice.
\end{proof}

\begin{claim}
\label{clm: minimal primes in support}
If $Q$ is a minimal prime in $\Supp(M/\fra M)$ such that $\dim(R/Q)=\dim(M/\fra M)$, then $Q=Q_i$ for some $i$.
\end{claim}
\begin{proof}[Proof of Claim]
Suppose on the contrary that $Q\neq Q_i$ for any $i$. Let $F_\bullet$ be a minimal free resolution of $M_Q$ over $R_Q$ and let $G_\bullet = F_\bullet \otimes_{R_Q} R_Q/\fra R_Q$. Then we know that $\Supp(G_\bullet)=\{QR_Q\}$ and thus by the new intersection theorem, $\text{length}(G_\bullet)\geq \dim(R_Q/\fra R_Q)$. On the other hand, we have that 
$$\depth(R_Q) \geq \pd(M_Q) = \text{legnth}(F_\bullet)=\text{legnth}(G_\bullet) \geq \dim(R_Q/\fra R_Q) \geq \depth(R_Q)$$
where the last inequality follows from the fact that $\fra$ is contained in some minimal prime of $R_Q$. It follows that we must have equalities throughout and in particular, we have 
$$\dim(R_Q/\fra R_Q)=\pd(M_Q).$$ Since $Q$ is a minimal prime in $\Supp(M/\fra M)$, $Q$ must contain a minimal prime $P$ in $\Supp(M)$. Since $Q\neq Q_i$ for any $i$, we must have $\height(P)\geq \dim(R)-\dim(M)$ by our assumption. Thus
\begin{align*}
\dim(R/\fra) & \geq \dim(R/Q) + \dim(R_Q/\fra R_Q) \\
& = \dim(M/\fra M) + \pd(M_Q) \\
& \geq \dim(M/ \fra M) + \pd(M_P)\\
& = \dim(M/ \fra M) + \height(P) \\
& \geq  \dim(M/\fra M) + \dim(R)-\dim(M)
\end{align*}
where for the second equality above we used the fact that $M_P$ is a module of finite length and finite projective dimension over $R_P$ and thus $R_P$ is Cohen--Macaulay (so $\pd(M_P)=\dim(R_P)=\height(P)$). But this contradicts Claim~\ref{clm: dimension inequality}.
\end{proof}

We now come back to the proof of the theorem. We let $W$ be the multiplicative set $R - \cup_{1\leq i\leq s}Q_i$ and set 
$$\fra := \ker(R\to W^{-1}R).$$
Since $\fra W^{-1}R = 0$ by construction, it follows that the minimal primes of $\fra$ are exactly those minimal primes of $R$ that are contained in some of the $Q_i$'s. Let $c:=\dim(M/\fra M)$ and choose $x_1,\dots,x_c$ a system of parameters on $M/\fra M$ that is also part of a system of parameters on $R$ and $R/\fra$. Consider 
$$\chi\big(M\otimes_R^\mathbb{L}R/\fra \otimes_R^\mathbb{L} \Kos(x_1,\dots,x_c; R)\big).$$
We will argue that the above quantity is both zero and nonzero and thus arrive at a contradiction. On the one hand, we have 
\begin{align*}
\dim(M) + \dim(R/(\fra, x_1,\dots,x_c)) & = \dim(M) + \dim(R/\fra) -c \\
& = \dim(M) + \dim(R/\fra) -\dim(M/\fra M)\\
& <\dim(R)
\end{align*}
where the last inequality follows from Claim~\ref{clm: dimension inequality}. Since each cohomology of $R/\fra \otimes_R^{\mathbb{L}}\Kos(x_1,\dots,x_c; R)$ is annihilated by $\fra+(x_1,\dots,x_c)$, by our assumption on the vanishing of $\chi$, we have that 
$$\chi\big(M\otimes_R^\mathbb{L}R/\fra \otimes_R^\mathbb{L} \Kos(x_1,\dots,x_c; R)\big)=0.$$
On the other hand, the zeroth cohomology of $M\otimes_R^\mathbb{L}R/\fra$ is $M/\fra M$, and by Claim~\ref{clm: minimal primes in support} we know that the minimal primes in $\Supp(M/\fra M)$ that has dimension $c$ are among $Q_1,\dots,Q_s$, and we have $\Tor_{>0}^R(M, R/\fra)_{Q_i}=0$ since $\fra R_{Q_i}=0$ by our construction. It follows that 
\begin{align*}
\chi\big(M\otimes_R^\mathbb{L}R/\fra \otimes_R^\mathbb{L} \Kos(x_1,\dots,x_c; R)\big) & =\chi\big(M/\fra M \otimes_R^\mathbb{L} \Kos(x_1,\dots,x_c; R)\big) \\
& = e((x_1,\dots,x_c), M/\fra M)\\
& \neq 0
\end{align*}
where the first equality uses the vanishing of $\chi(N \otimes_R^\mathbb{L} \Kos(x_1,\dots,x_c; R))$ when $\dim(N)<c$. Thus we have arrived at the desired contradiction.
\end{proof}

As a consequence, finitely generated modules of finite projective dimension that are liftable to regular local rings satisfies Conjecture~\ref{conj: Peskine--Szpiro}.

\begin{corollary}
\label{cor: liftable}
Let $R$ be a Noetherian local ring that admits a local map from a regular local ring $S$. Suppose there is a finitely generated $S$-module $L$ so that $M:= L \otimes^\mathbb{L}_SR$ is discrete (in particular $M$ is an $R$-module of finite projective dimension). Then all three statements of Conjecture~\ref{conj: Peskine--Szpiro} hold for $M$. 
\end{corollary}
\begin{proof}
It is enough to prove the dimension inequality and the grade conjecture for $M$. Replacing $R, S, M, L$ by their completions, we may assume $R$ and $S$ are complete local rings. By Cohen factorization \cite{AvramovFoxbyStructureLocalHomomorphisms}, we may assume that $S\to R$ is surjective. By \cite[Proposition 1.6]{KCSotoLevinsLifting}, we know that 
\begin{equation}
\label{eqn: dimension equality}
    \dim(L)-\dim(M) = \dim(S)-\dim(R).
\end{equation}

Now if $M\otimes_R N$ has finite length, then so is $L\otimes_SN$ and thus by Serre's dimension inequality over the regular local ring $S$ \cite{SerreLocalAlgebra}, we know that $\dim(L)+\dim(N)\leq \dim(S).$
This combined with (\ref{eqn: dimension equality}) shows that 
$$\dim(M) + \dim(N)\leq \dim(R).$$

To prove the grade conjecture for $M$, by Theorem~\ref{thm: vanishing implies grade}, it is enough to show that $\chi(M, R/J)=0$ whenever $M\otimes_R R/J$ has finite length and $\dim(M) + \dim(R/J)<\dim(R)$. But note that 
$$M \otimes_R^\mathbb{L}R/J = L \otimes_S^\mathbb{L}R/J$$
and $\dim(L) + \dim(R/J)  <\dim(S)$ by (\ref{eqn: dimension equality}). Thus by the vanishing of Serre intersection multiplicity over the regular local ring $S$ \cite{RobertsVanshing,GilletSouleIntersectionAdamsOperations}, we know that $\chi(M \otimes_R^\mathbb{L}R/J) = \chi(L \otimes_S^\mathbb{L}R/J)=0$ as wanted.
\end{proof}

Lastly, note that in the counterexamples to the dimension inequality (resp., the grade conjecture), the ring $R$ has dimension three (resp., four). These are indeed the smallest possible dimensions of $R$ in  potential counterexamples, as the following proposition shows. 

\begin{proposition}
\label{prop: small dimensions}
Let $R$ be a Noetherian local ring of dimension $d$ and let $M$ be a finitely generated $R$-module of finite projective dimension. Suppose $N$ is a finitely generated $R$-module so that $M\otimes_RN$ has finite length. Then 
\begin{enumerate}
    \item If $d\leq 2$, then $\dim(M) + \dim(N) \leq d$.
    \item If $d\leq 3$, the $\grade(M) = d - \dim(M)$.
\end{enumerate}
\end{proposition}
\begin{proof}
We first prove $(1)$ and we will assume $d=2$ (the case $d\leq 1$ is much easier and we leave it to the readers). If $\dim(M)=0$ then the conclusion is obvious. If $\dim(M)=1$, then $\Ann_R(M)\neq 0$ and thus contains a nonzerodivisor, see \cite[Corollary 20.13]{EisenbudCABook}. It follows that $$1\leq \grade(M)\leq \dim(R)-\dim(M)=1$$ and thus $\grade(M)=1$. Now by \cite[Lemma (0.1) (d) and (e)]{FoxbyMacRaeInvariant}, we know that $V(z)\subseteq \Supp(M)$ for some nonzerodivisor $z\in R$. This implies that 
$$\Supp(N) \cap V(z)\subseteq \Supp(N) \cap \Supp(M)=\{\m\}$$ and thus $N/zN$ has finite length. Hence $\dim(N)\leq 1 = d-\dim(M)$. Finally, if $\dim(M)=2$, then $\Ann_R(M)=0$ (as otherwise it contains a nonzerodivisor by \cite[Corollary 20.13]{EisenbudCABook}). This means $\Supp(M)=\Spec(R)$ and thus $N$ must have finite length by assumption, so $\dim(M)+\dim(N)=d$ in this case.

We next prove $(2)$ and we will assume $d=3$ (the case $d\leq 2$ is easier and we leave it to the readers). If $\dim(M)=0$, then $M$ has finite length and finite projective dimension, thus $R$ is Cohen--Macaulay and $\grade(M)=3=d-\dim(M)$. If $\dim(M)=1$, then $\grade(M)\leq d-\dim(M)=2$. Note that $\grade(M)\neq 0$, for otherwise $\Ann_R(M)=0$ by \cite[Corollary 20.13]{EisenbudCABook} which contradicts $\dim(M)=1$. If $\grade(M)=1$, then by \cite[Lemma (0.1) (d) and (e)]{FoxbyMacRaeInvariant} we know that $V(z)\subseteq \Supp(M)$ for some nonzerodivisor $z\in R$ and thus $\dim(M)\geq \dim(V(z))=2$ which is a contradiction. Hence $\grade(M)=2=d-\dim(M)$ in this case. If $\dim(M)=2$, then $\grade(M)\leq d-\dim(M)=1$ and we know $\grade(M)\neq 0$ (again by \cite[Corollary 20.13]{EisenbudCABook}) so we get equality. Finally, the case $\dim(M)=3$ is obvious as this forces $\grade(M)=0$. 
\end{proof}

\newpage
\bibliographystyle{alpha}
\bibliography{Bib}

@article {DuttaHochsterMcLaughlin,
    AUTHOR = {Dutta, Sankar P. and Hochster, M. and McLaughlin, J. E.},
     TITLE = {Modules of finite projective dimension with negative
              intersection multiplicities},
   JOURNAL = {Invent. Math.},
  FJOURNAL = {Inventiones Mathematicae},
    VOLUME = {79},
      YEAR = {1985},
    NUMBER = {2},
     PAGES = {253--291},
      ISSN = {0020-9910,1432-1297},
   MRCLASS = {13H15 (13C15 14C17)},
  MRNUMBER = {778127},
MRREVIEWER = {Luchezar\ L.\ Avramov},
       DOI = {10.1007/BF01388973},
       URL = {https://doi.org/10.1007/BF01388973},
}

@article {KCSotoLevinsLifting,
    AUTHOR = {KC, Nawaj and Soto Levins, Andrew J.},
     TITLE = {On liftings of modules of finite projective dimension},
   JOURNAL = {Int. Math. Res. Not. IMRN},
  FJOURNAL = {International Mathematics Research Notices. IMRN},
      YEAR = {2024},
    NUMBER = {24},
     PAGES = {14729--14736},
      ISSN = {1073-7928,1687-0247},
   MRCLASS = {13C15 (13H05)},
  MRNUMBER = {4843654},
MRREVIEWER = {Mahdi\ Majidi-Zolbanin},
       DOI = {10.1093/imrn/rnae257},
       URL = {https://doi.org/10.1093/imrn/rnae257},
}

@article {AvramovFoxbyStructureLocalHomomorphisms,
    AUTHOR = {Avramov, Luchezar L. and Foxby, Hans-Bj{\o}rn and Herzog,
              Bernd},
     TITLE = {Structure of local homomorphisms},
   JOURNAL = {J. Algebra},
  FJOURNAL = {Journal of Algebra},
    VOLUME = {164},
      YEAR = {1994},
    NUMBER = {1},
     PAGES = {124--145},
      ISSN = {0021-8693,1090-266X},
   MRCLASS = {13H99},
  MRNUMBER = {1268330},
MRREVIEWER = {L\^{e}\ Tu\^{a}n\ Hoa},
       DOI = {10.1006/jabr.1994.1057},
       URL = {https://doi.org/10.1006/jabr.1994.1057},
}

@article {PeskineSzpiroSyzygiesMultiplicities,
    AUTHOR = {Peskine, Christian and Szpiro, Lucien},
     TITLE = {Syzygies et multiplicit\'{e}s},
   JOURNAL = {C. R. Acad. Sci. Paris S\'{e}r. A},
  FJOURNAL = {Comptes Rendus Hebdomadaires des S\'{e}ances de l'Acad\'{e}mie
              des Sciences. S\'{e}rie A. Sciences Math\'{e}matiques},
    VOLUME = {278},
      YEAR = {1974},
     PAGES = {1421--1424},
      ISSN = {0302-8429},
   MRCLASS = {13D05 (13H15)},
  MRNUMBER = {349659},
MRREVIEWER = {Y.\ Nakai},
}

@book {EisenbudCABook,
    AUTHOR = {Eisenbud, David},
     TITLE = {Commutative algebra},
    SERIES = {Graduate Texts in Mathematics},
    VOLUME = {150},
      NOTE = {With a view toward algebraic geometry},
 PUBLISHER = {Springer-Verlag, New York},
      YEAR = {1995},
     PAGES = {xvi+785},
      ISBN = {0-387-94268-8; 0-387-94269-6},
   MRCLASS = {13-01 (14A05)},
  MRNUMBER = {1322960},
MRREVIEWER = {Matthew\ Miller},
       DOI = {10.1007/978-1-4612-5350-1},
       URL = {https://doi.org/10.1007/978-1-4612-5350-1},
}

@book {FultonIntersectionTheoryOriginal,
    AUTHOR = {Fulton, William},
     TITLE = {Intersection theory},
    SERIES = {Ergebnisse der Mathematik und ihrer Grenzgebiete (3) [Results
              in Mathematics and Related Areas (3)]},
    VOLUME = {2},
 PUBLISHER = {Springer-Verlag, Berlin},
      YEAR = {1984},
     PAGES = {xi+470},
      ISBN = {3-540-12176-5},
   MRCLASS = {14C17 (14-02 14C40)},
  MRNUMBER = {732620},
MRREVIEWER = {Werner\ Kleinert},
       DOI = {10.1007/978-3-662-02421-8},
       URL = {https://doi.org/10.1007/978-3-662-02421-8},
}

@book {WeibelKbook,
    AUTHOR = {Weibel, Charles A.},
     TITLE = {The {$K$}-book},
    SERIES = {Graduate Studies in Mathematics},
    VOLUME = {145},
      NOTE = {An introduction to algebraic $K$-theory},
 PUBLISHER = {American Mathematical Society, Providence, RI},
      YEAR = {2013},
     PAGES = {xii+618},
      ISBN = {978-0-8218-9132-2},
   MRCLASS = {19-01},
  MRNUMBER = {3076731},
MRREVIEWER = {L.\ N.\ Vaserstein},
       DOI = {10.1090/gsm/145},
       URL = {https://doi.org/10.1090/gsm/145},
}

@article {BenBassatBlockMilnorDescent,
    AUTHOR = {Ben-Bassat, Oren and Block, Jonathan},
     TITLE = {Milnor descent for cohesive dg-categories},
   JOURNAL = {J. K-Theory},
  FJOURNAL = {Journal of K-Theory. K-Theory and its Applications in Algebra,
              Geometry, Analysis \& Topology},
    VOLUME = {12},
      YEAR = {2013},
    NUMBER = {3},
     PAGES = {433--459},
      ISSN = {1865-2433,1865-5394},
   MRCLASS = {18D99 (46L87 58B34)},
  MRNUMBER = {3165183},
MRREVIEWER = {Pawel\ Sosna},
       DOI = {10.1017/is013007003jkt236},
       URL = {https://doi.org/10.1017/is013007003jkt236},
}

@book {RobertsBook,
    AUTHOR = {Roberts, Paul C.},
     TITLE = {Multiplicities and {C}hern classes in local algebra},
    SERIES = {Cambridge Tracts in Mathematics},
    VOLUME = {133},
 PUBLISHER = {Cambridge University Press, Cambridge},
      YEAR = {1998},
     PAGES = {xii+303},
      ISBN = {0-521-47316-0},
   MRCLASS = {13D22 (13H15 14C17)},
  MRNUMBER = {1686450},
MRREVIEWER = {Hans-Bj\o rn\ Foxby},
       DOI = {10.1017/CBO9780511529986},
       URL = {https://doi.org/10.1017/CBO9780511529986},
}

@article {GilletSouleIntersectionAdamsOperations,
    AUTHOR = {Gillet, H. and Soul\'{e}, C.},
     TITLE = {Intersection theory using {A}dams operations},
   JOURNAL = {Invent. Math.},
  FJOURNAL = {Inventiones Mathematicae},
    VOLUME = {90},
      YEAR = {1987},
    NUMBER = {2},
     PAGES = {243--277},
      ISSN = {0020-9910,1432-1297},
   MRCLASS = {14C17 (13D05 14B15 14C40 19E08)},
  MRNUMBER = {910201},
MRREVIEWER = {G.\ Horrocks},
       DOI = {10.1007/BF01388705},
       URL = {https://doi.org/10.1007/BF01388705},
}

@article {RobertsVanshing,
    AUTHOR = {Roberts, Paul},
     TITLE = {The vanishing of intersection multiplicities of perfect
              complexes},
   JOURNAL = {Bull. Amer. Math. Soc. (N.S.)},
  FJOURNAL = {American Mathematical Society. Bulletin. New Series},
    VOLUME = {13},
      YEAR = {1985},
    NUMBER = {2},
     PAGES = {127--130},
      ISSN = {0273-0979,1088-9485},
   MRCLASS = {13H15 (13D30 14C17)},
  MRNUMBER = {799793},
MRREVIEWER = {Marc\ Levine},
       DOI = {10.1090/S0273-0979-1985-15394-7},
       URL = {https://doi.org/10.1090/S0273-0979-1985-15394-7},
}

@article{RobertsSrinivasModulesoffinitelengthandfiniteprojectivedimension,
    AUTHOR = {Roberts, Paul C. and Srinivas, V.},
     TITLE = {Modules of finite length and finite projective dimension},
   JOURNAL = {Invent. Math.},
  FJOURNAL = {Inventiones Mathematicae},
    VOLUME = {151},
      YEAR = {2003},
    NUMBER = {1},
     PAGES = {1--27},
      ISSN = {0020-9910,1432-1297},
   MRCLASS = {13D22 (14C17)},
  MRNUMBER = {1943740},
MRREVIEWER = {Keri\ Sather-Wagstaff},
       DOI = {10.1007/s002220200217},
       URL = {https://doi.org/10.1007/s002220200217},
}

@incollection {FoxbyMacRaeInvariant,
    AUTHOR = {Foxby, Hans-Bj{\o}rn},
     TITLE = {The {M}ac{R}ae invariant},
 BOOKTITLE = {Commutative algebra: {D}urham 1981 ({D}urham, 1981)},
    SERIES = {London Math. Soc. Lecture Note Ser.},
    VOLUME = {72},
     PAGES = {121--128},
 PUBLISHER = {Cambridge Univ. Press, Cambridge},
      YEAR = {1982},
      ISBN = {0-521-27125-8},
   MRCLASS = {13H15 (13D15 13D25)},
  MRNUMBER = {693631},
}

@article {ChanHuang,
    AUTHOR = {Chan, C.-Y. Jean and Huang, I-Chiau},
     TITLE = {Module structure of an injective resolution},
   JOURNAL = {Comm. Algebra},
  FJOURNAL = {Communications in Algebra},
    VOLUME = {35},
      YEAR = {2007},
    NUMBER = {11},
     PAGES = {3713--3750},
      ISSN = {0092-7872,1532-4125},
   MRCLASS = {13D02 (13D07 13D45)},
  MRNUMBER = {2362680},
MRREVIEWER = {Reza\ Sazeedeh},
       DOI = {10.1080/00927870701404770},
       URL = {https://doi.org/10.1080/00927870701404770},
}

@article {HeitmannCounterexampleRigidityofTor,
    AUTHOR = {Heitmann, Raymond C.},
     TITLE = {A counterexample to the rigidity conjecture for rings},
   JOURNAL = {Bull. Amer. Math. Soc. (N.S.)},
  FJOURNAL = {American Mathematical Society. Bulletin. New Series},
    VOLUME = {29},
      YEAR = {1993},
    NUMBER = {1},
     PAGES = {94--97},
      ISSN = {0273-0979,1088-9485},
   MRCLASS = {13D05 (18G15)},
  MRNUMBER = {1197425},
MRREVIEWER = {Jos\'{e}\ A.\ Hermida-Alonso},
       DOI = {10.1090/S0273-0979-1993-00410-5},
       URL = {https://doi.org/10.1090/S0273-0979-1993-00410-5},
}

@book {SerreLocalAlgebra,
    AUTHOR = {Serre, Jean-Pierre},
     TITLE = {Alg\`ebre locale. {M}ultiplicit\'{e}s},
    SERIES = {Lecture Notes in Mathematics},
    VOLUME = {11},
      NOTE = {Cours au Coll\`ege de France, 1957--1958, r\'{e}dig\'{e} par
              Pierre Gabriel,
              Seconde \'{e}dition, 1965},
 PUBLISHER = {Springer-Verlag, Berlin-New York},
      YEAR = {1965},
     PAGES = {vii+188 pp. (not consecutively paged)},
   MRCLASS = {13.95 (14.08)},
  MRNUMBER = {201468},
MRREVIEWER = {M.\ Nagata},
}

@article {AlbertMuckenhouptMatricesTraceZero,
    AUTHOR = {Albert, A. A. and Muckenhoupt, Benjamin},
     TITLE = {On matrices of trace zeros},
   JOURNAL = {Michigan Math. J.},
  FJOURNAL = {Michigan Mathematical Journal},
    VOLUME = {4},
      YEAR = {1957},
     PAGES = {1--3},
      ISSN = {0026-2285,1945-2365},
   MRCLASS = {15.0X},
  MRNUMBER = {83961},
MRREVIEWER = {D.\ E.\ Rutherford},
       URL = {http://projecteuclid.org/euclid.mmj/1028990168},
}

@article {AndreDSC,
    AUTHOR = {Andr\'{e}, Yves},
     TITLE = {La conjecture du facteur direct},
   JOURNAL = {Publ. Math. Inst. Hautes \'{E}tudes Sci.},
  FJOURNAL = {Publications Math\'{e}matiques. Institut de Hautes \'{E}tudes
              Scientifiques},
    VOLUME = {127},
      YEAR = {2018},
     PAGES = {71--93},
      ISSN = {0073-8301,1618-1913},
   MRCLASS = {13D22 (13A35 13B40 13D09 18A99)},
  MRNUMBER = {3814651},
MRREVIEWER = {Marcel\ Morales},
       DOI = {10.1007/s10240-017-0097-9},
       URL = {https://doi.org/10.1007/s10240-017-0097-9},
}

@article {AvramovFoxbyCMdefect,
    AUTHOR = {Avramov, Luchezar L. and Foxby, Hans-Bj{\o}rn},
     TITLE = {Cohen-{M}acaulay properties of ring homomorphisms},
   JOURNAL = {Adv. Math.},
  FJOURNAL = {Advances in Mathematics},
    VOLUME = {133},
      YEAR = {1998},
    NUMBER = {1},
     PAGES = {54--95},
      ISSN = {0001-8708,1090-2082},
   MRCLASS = {13H10 (13B10 14E40)},
  MRNUMBER = {1492786},
MRREVIEWER = {Rafael\ H.\ Villarreal},
       DOI = {10.1006/aima.1997.1684},
       URL = {https://doi.org/10.1006/aima.1997.1684},
}

@book {HochsterTopicsCBMS,
    AUTHOR = {Hochster, Melvin},
     TITLE = {Topics in the homological theory of modules over commutative
              rings},
    SERIES = {Conference Board of the Mathematical Sciences Regional
              Conference Series in Mathematics, No. 24},
 PUBLISHER = {Published for the Conference Board of the Mathematical               Sciences, Washington, DC; by American Mathematical Society,               Providence, RI},
      NOTE = {Expository lectures from the CBMS Regional Conference held at the University of Nebraska, Lincoln, Neb., June 24--28, 1974},
      YEAR = {1975},
     PAGES = {vii+75},
   MRCLASS = {13DXX (14BXX)},
  MRNUMBER = {371879},
MRREVIEWER = {Tadayuki\ Matsuoka},
}

@incollection {HochsterDimensionIntersectionHypersurface,
    AUTHOR = {Hochster, Melvin},
     TITLE = {The dimension of an intersection in an ambient hypersurface},
 BOOKTITLE = {Algebraic geometry ({C}hicago, {I}ll., 1980)},
    SERIES = {Lecture Notes in Math.},
    VOLUME = {862},
     PAGES = {93--106},
 PUBLISHER = {Springer, Berlin},
      YEAR = {1981},
      ISBN = {3-540-10833-5},
   MRCLASS = {13H15 (13D15 14B25)},
  MRNUMBER = {644818},
MRREVIEWER = {L.\ B\u{a}descu},
}

@article {RobertsIntersectionTheorem,
    AUTHOR = {Roberts, Paul},
     TITLE = {Le th\'{e}or\`eme d'intersection},
   JOURNAL = {C. R. Acad. Sci. Paris S\'{e}r. I Math.},
  FJOURNAL = {Comptes Rendus des S\'{e}ances de l'Acad\'{e}mie des Sciences.
              S\'{e}rie I. Math\'{e}matique},
    VOLUME = {304},
      YEAR = {1987},
    NUMBER = {7},
     PAGES = {177--180},
      ISSN = {0249-6291},
   MRCLASS = {14C17 (13D25)},
  MRNUMBER = {880574},
MRREVIEWER = {R\"{u}diger\ Achilles},
}

@incollection {ThomasonTrobaughKtheory,
    AUTHOR = {Thomason, R. W. and Trobaugh, Thomas},
     TITLE = {Higher algebraic {$K$}-theory of schemes and of derived
              categories},
 BOOKTITLE = {The {G}rothendieck {F}estschrift, {V}ol. {III}},
    SERIES = {Progr. Math.},
    VOLUME = {88},
     PAGES = {247--435},
 PUBLISHER = {Birkh\"{a}user Boston, Boston, MA},
      YEAR = {1990},
      ISBN = {0-8176-3487-8},
   MRCLASS = {19E08 (14C35 19D10)},
  MRNUMBER = {1106918},
MRREVIEWER = {Charles\ Weibel},
       DOI = {10.1007/978-0-8176-4576-2\{_}10}

@article {KuranoNumericalequivalenceonChowgroupsoflocalrings,
    AUTHOR = {Kurano, Kazuhiko},
     TITLE = {Numerical equivalence defined on {C}how groups of {N}oetherian
              local rings},
   JOURNAL = {Invent. Math.},
  FJOURNAL = {Inventiones Mathematicae},
    VOLUME = {157},
      YEAR = {2004},
    NUMBER = {3},
     PAGES = {575--619},
      ISSN = {0020-9910,1432-1297},
   MRCLASS = {14C15 (13D15 13D40)},
  MRNUMBER = {2092770},
MRREVIEWER = {Gerhard\ Pfister},
       DOI = {10.1007/s00222-004-0361-8},
       URL = {https://doi.org/10.1007/s00222-004-0361-8},
}

@article {PeskineSzpiroIHES,
    AUTHOR = {Peskine, C. and Szpiro, L.},
     TITLE = {Dimension projective finie et cohomologie locale.
              {A}pplications \`a la d\'{e}monstration de conjectures de {M}.
              {A}uslander, {H}. {B}ass et {A}. {G}rothendieck},
   JOURNAL = {Inst. Hautes \'{E}tudes Sci. Publ. Math.},
  FJOURNAL = {Institut des Hautes \'{E}tudes Scientifiques. Publications
              Math\'{e}matiques},
    NUMBER = {42},
      YEAR = {1973},
     PAGES = {47--119},
      ISSN = {0073-8301,1618-1913},
   MRCLASS = {14B15 (13D05 13H10)},
  MRNUMBER = {374130},
MRREVIEWER = {Melvin\ Hochster},
       URL = {http://www.numdam.org/item?id=PMIHES_1973__42__47_0},
}
\end{document}